\documentclass{amsart}
\usepackage{amsfonts,amssymb,amsmath}
\usepackage{enumerate}
\usepackage{braket}

\def\F{\mathbb{F}}

\def\N{\mathbb{N}}

\def\Z{\mathbb{Z}}

\newcommand{\CalC}{\mathcal{C}}
\newcommand{\CalO}{\mathcal{O}}

\newtheorem{theorem}{Theorem}[section] % 1st argument is your name for it
\newtheorem{lemma}[theorem]{Lemma}     % 2nd argument is what is printed
\newtheorem{corollary}[theorem]{Corollary}
\newtheorem{proposition}[theorem]{Proposition}
\newtheorem{remark}[theorem]{Remark}
\newcommand{\Hom}{\operatorname{Hom}}

\newcommand{\zp}{\Z_p}

\newcommand{\notty}{\mathfrak{N}}

\newcommand{\powerseries}{\F_p[\![t ]\!]}
\newcommand{\kappapower}{\kappa[\![t ]\!]}
\newcommand{\laurentseries}{\F_p(\!(t )\!)}
\newcommand{\kappaseries}{\kappa(\!(t )\!)}

\newcommand{\maxi}{\maxideal}

\newcommand{\npg}{U_1}
\newcommand{\maxideal}{\mathfrak{M}}

\newcommand{\zmodpnz}{\Z/p^n\Z}

\newcommand{\FZ}{\mathfrak{Z}}
\newcommand{\FM}{\mathfrak{M}}
\newcommand{\pmodp}{\pmod{p}}

\newcommand{\ub}[1]{{}_{#1}}

\newcommand{\wt}[1]{\widetilde{#1}}

\newtheorem{innercustomthm}{Theorem}

\newenvironment{customthm}[1]
  {\renewcommand\theinnercustomthm{#1}\innercustomthm}
  {\endinnercustomthm}
  
\begin{document}

\title{Torsion elements of the Nottingham group of order $p^2$}
\author{Chun Yin Hui and Krishna Kishore}
\thanks{C.Y. Hui is supported by China's Thousand Talents Plan: The Recruitment Program for Young Professionals.}

\maketitle

\begin{abstract}
We establish an explicit upper bound $B(p,l,m)$, depending on $p,l,m$, on the number of conjugacy classes of order $p^2$ torsion elements $u$ of type $\langle l,m\rangle$ of the Nottingham group defined over the prime field of characteristic $p >0$. In the cases where $l < p $, the number of conjugacy classes
of type $\langle l,m\rangle$  coincides with $B(p,l,m)$.
Moreover, we give a criterion on when
$u$ and $u^n$ are conjugate.
\end{abstract}

\section{Introduction}\label{intro}
This paper is a continuation of \cite{Ki} where the second author classified elements of the Nottingham group over the \textit{prime} field $\F_p$ that are of order $p^2$ and type $\langle 2,m \rangle$. In order to discuss some preliminary notions and to state our main results we need some notation. Let $\kappa$ be a finite field of characteristic $p > 0$,  let  $K := \kappa(\!( t )\!)$ be the field of Laurent-series in one variable $t$ over 
%the prime field 
$\kappa$ and $\CalO_K$ its ring of integers, and $\maxideal$ the maximal ideal of $\CalO_K$ generated by the uniformizer
$t$. Consider the group of automorphisms of $K$ that fix $\kappa$, and in turn consider the subgroup of wild automorphisms namely those which map the uniformizer $t$ to the product $tz$ for some $z$ in the principal unit group $U_1 := 1 +  \maxideal$. The set $\Set{t z | z \in U_1}$ corresponding to this subgroup equipped with composition of power series is a group, called the \textit{Nottingham group} $\notty_\kappa$ over $\kappa$; for a detailed description about the Nottingham group see \cite{Ca} and \cite{Ca1}.

The classification of elements of order $p$ over $\kappa$ is due to Klopsch \cite{Kl}. On the other hand, by associating conjugacy classes of torsion elements of order $p^n$ in $\notty_\kappa$  with continuous surjective characters $\chi: U_1 \to \Z/p^n \Z$ up to \emph{strict equivalence},  Lubin \cite{Lu} deduced the result of Klopch as a particular case, but most importantly constructed a framework to understand torsion elements of high order. This is due to the fact that each surjective $\chi$, in turn, is associated with a sequence of integers $\langle b^{(0)}, \ldots, b^{(n-1)} \rangle$, called \textit{the break sequence of $\chi$ or the \emph{type} of $\chi$}, where $b^{(j-1)}$ is defined as the largest positive integer $b$ such that 
$$
\chi(1 + \maxi^b) \not \subset p^j \Z/p^n \Z.
$$
Since the type of $\chi$ is invariant under strict equivalence,
the analysis of the conjugacy classes of torsion elements of  order $p^n$ reduces to the analysis of the strict equivalence classes of a given type.

Now, let us restrict our attention to the case where $\kappa = \F_p$, the prime finite field of characteristic $p > 0$. We drop the subscripts $\kappa$ from the notation, for instance instead of $\notty_{\F_p}$ we simply denote it as $\notty$. The continuous characters $\Hom_{\Z_p}^{\textrm{cont}}(U_1, \Z/p^n \Z)$ form a $\Z_p$-module equipped with the action of the Nottingham group $\notty$ that acts on the left in a manner compatible with $\Z_p$-module structure, namely as  
$
{}_u \chi(f(t)) := \chi( f \circ u(t)) 
$
for $f(t) \in \npg = 1 + t \powerseries$ and $u \in \notty$.  Two such characters $\chi, \psi \in \Hom_{\Z_p}^{\textrm{cont}}(U_1, \Z/p^n \Z) $ are said to be \textit{strictly equivalent}, denoted $\chi \simeq \psi$, if there exists an element $u \in \notty$ such that $\psi = \ub{u} \chi$ and $u(t)/t \in \ker \chi$; if only the former condition holds then they are said to be \emph{weakly equivalent}, denoted  $\chi \sim \psi$. Both relations $\simeq$ and $\sim$ are equivalent relations on $\Hom_{\Z_p}^{\textrm{cont}}(U_1, \Z/p^n \Z)$; see \cite{Lu}. Let us note here that, given a surjective character $\chi\in \Hom_{\Z_p}^{\textrm{cont}}(U_1, \Z/p^n \Z)$, characters strictly (resp. weakly) equivalent to $\chi$ are surjective and are of the same type.

In an unpublished work, Lubin classified torsion elements of order $p^2$ of type $\langle 1, m \rangle$ over \textit{any} finite field; see Theorem $3.6$  in \cite{Ki}.  On the other hand, the second author classified weak equivalence classes of order $p^2$ of type $\langle 2,m \rangle$ over any prime finite field and gave bounds on the number of strict equivalent classes; see Theorem \ref{prior_thm_3}. The main goal of this article is to give bounds on the number of conjugacy classes of torsion elements of order $p^2$ of any type over any \textit{prime} finite field.

Let $\Set{E_k : = 1 + t^k}_{p \nmid k}$ be a topological basis of the principal unit group $U_1 = 1 + \FM$, and $\{\FZ_i\}_i$ be the basis of $\Hom_{\Z_p}^{\textrm{cont}}(1 + \FM, \Z/p^2 \Z)$ dual to $E_i$. Then any character $\chi \in \Hom_{\Z_p}^{\textrm{cont}}(1 + \FM, \Z/p^2 \Z)$ of type $\langle l, m \rangle$ has the following expansion, called the \emph{standard expansion}:
\begin{equation}\label{intro_eq}
\chi = \sum_{1 \leq i \leq l, p \nmid i} x_i \FZ_i  + \sum_{\substack{1 \leq j \leq m, p \nmid j}} a_j. p \FZ_j,
\end{equation}
where 
%all the coefficients are in $\Z/p^2 \Z$, which have representatives as 
$x_i, a_j$ belong to $\{0,1,...,p-1\}$ for all $i,j$. The coefficient $x_l$ is nonzero, and if $p \nmid m$ then $a_m$ is nonzero too. Now we can state some of our main results. \\

We show that every strict equivalence class contains a character of a special form. 

\begin{theorem}\label{main_thm}
Let $[\chi]$ be a strict equivalence class of type $\langle l, m \rangle$.  Then there exists a character in the class $[\chi]$ whose standard expansion is 
\begin{equation}\label{intro_reduction}
x_l \FZ_l + \sum_{ m-l \leq j \leq m, p \nmid j} b_j \cdot p \FZ_j, 
\end{equation}
where $x_l \neq 0$ and
if $p \nmid m$, then $b_m \neq 0$. 
The character in \eqref{intro_reduction} is said to be in \textit{reduced form}.
\end{theorem}

Let $B(p,l,m)$ be  the number of type $\langle l, m \rangle$ reduced forms \eqref{intro_reduction}.
In the case where $l < p$, 
$B(p,l,m)$ is equal to the number of strict equivalence classes of type $\langle l, m \rangle$.

\begin{theorem}\label{exact}
%Let $[\chi]$ be a strict equivalence class of type $\langle l, m \rangle$. 
Suppose $l < p$. Then each strict equivalence class of type $\langle l, m \rangle$ has precisely one representative in the reduced form \eqref{intro_reduction} in Theorem \ref{main_thm}.
\end{theorem}

The condition $l<p$  of Theorem \ref{exact} is necessary. We will provide a counterexample for $p=2$ and type $\langle 5, 15 \rangle$. 
Finally, by combining
Theorems \ref{main_thm}, \ref{exact},
and Lubin \cite{Lu},
we obtain our main result on the torsion elements of $\notty$ of order $p^2$.

%Finally, our main theorem asserts that the conjugacy classes of torsion elements of Nottingham group may be classified according to type $\langle l , m \rangle$ (this is due to Lubin), and for each such type the theorem provides an upper bound on the number of conjugacy classes.

\begin{theorem}\label{intro_final_thm}
Let $d_{l,m}$ denote the number of conjugacy classes of elements of the Nottingham group $\notty$ that are of order $p^2$ and type $\langle l, m \rangle$. Then
\begin{equation}\label{ineq}
d_{l,m} \leq B(p,l,m)=p^k (p-1)^\epsilon
\end{equation}
where $k$ is the number of integers in $[m-l,m-1]$ incongruent to $0$ modulo $p$, and $\epsilon$ is equal to $1$ if $m$ is congruent to $0$ modulo $p$, and equal to $2$ otherwise. The inequality \eqref{ineq} is an equality if $l<p$.
\end{theorem}

Let $n\in\N$ and $u\in\notty$ be a torsion element of order $p$. It is not difficult to see (e.g., by  Theorem \ref{prior_thm_1})
that $u$ and $u^n$ are conjugate in $\notty$ if and only if $u=u^n$.
As an application of the above results, the following group theoretic result related to the Nottingham group is of independent interest.

\begin{theorem}\label{conjugate}
Let $n\in\N$ and $u\in\notty$ be a torsion element of order $p^2$ and type $\langle l,m\rangle$. If
$u$ and $u^n$ are distinct elements,
then they are conjugate in $\notty$ if and only if $n\equiv 1 \pmodp$ and $(p,l,m)\neq (2,l,2l)$.
\end{theorem}

The organization of the paper is as follows. In \S \ref{not} we establish the notation and conventions that will be adapted throughout the paper, and also state some prior results so that the reader may appreciate the results established in this paper better. In \S \ref{prelim_comp} we perform some preliminary computations that will be useful in \S \ref{classification}. In \S \ref{classification} we prove Theorems \ref{main_thm} and \ref{exact} and also justify the condition $l < p$ by providing a counterexample. In \S \ref{torsion_elements} we prove our main results (Theorems \ref{intro_final_thm} and \ref{conjugate}) on torsion elements of the Nottingham group $\notty$ of order $p^2$.

\section{Notation, conventions, and prior results}\label{not}

We adapt the following notation in this article. The letter $p$ always denotes a prime number, and $\F_p$ the prime field of characteristic $p >0$. By $\kappa$ we mean a characteristic $p$ finite field, $K$ is the field of Laurent-series  $\kappaseries$ over $\kappa$, and $\CalO_K := \kappapower$ its ring of integers with maximal ideal $\maxideal = t \CalO_K$ where $t$ is a fixed uniformizer of $\CalO_K$. Then the group of principal units $ 1 + \maxideal $ is denoted by $U_1$, and its higher unit  subgroups $1 + \maxideal^j$, $j \geq 1$, by $U_j$.

The Nottingham group over $\kappa$, denoted by  $\notty_\kappa$, is the subgroup of elements of $\mathrm{Aut}_\kappa(K)$ mapping $t$ to $u(t)\in t(1+\maxideal)$. An element of $\notty_\kappa$ is determined and represented by the power series $u(t)$.   
We drop the subscript $\kappa$ in the case where $\kappa = \F_p$, the prime finite field of characteristic $p > 0$. So, for example $\notty_{\F_p}$ is simply written as $\notty$. 

In this article, we deal \textit{only} with continuous characters $\chi : U_1 \to \Z/p^n \Z$, where $U_1$ equipped with induced topology from that of the topological group $K^\times$, and $\Z/p^n \Z$ with the discrete topology.  So we omit the adjective `continuous' while referring to the characters in the $\zp$-module $\Hom_{\Z_p}^{\textrm{cont}}(\npg, \Z/p^n \Z)$. 

The definitions of strict equivalence $\simeq$ (resp. weak equivalence $\sim$) on  $\Hom_{\Z_p}^{\textrm{cont}}(\npg, \Z/p^n \Z)$ and the break sequence (or the type) of a surjective character are defined in the same ways as the $\kappa=\F_p$ case (see $\mathsection1$) so that a strictly (resp. weakly) equivalent class of surjective characters have the same type.  When we say a character has certain type, the character is assumed to be a surjective character.

We now describe the main results of Lubin on torsion elements of $\notty_\kappa$.
Given an order $p^n$ element $u\in\notty_\kappa\subset\mathrm{Aut}_\kappa(K)$, the subset $F$ of  $K=\kappaseries$ fixed by $u$ is a subfield of $K$ for which $K/F$ is a degree $p^n$ totally ramified abelian extension.
Local class field theory produces a canonical continuous group homomorphism 
onto the subgroup generated by $u$:
$$\rho^K_F:F^*\to \langle u\rangle.$$
Since $K/F$ is totally ramified, the norm $N_{K/F}(t)$ of $t$ is a uniformizer of $F$. Hence, there is a unique $\kappa$-isomorphism $K\cong F$ of fields mapping $t$ to $N_{K/F}(t)$. On the other hand, there is a unique group isomorphism $\langle u\rangle\cong \Z/p^n\Z$ mapping $u$ to $1$.
Therefore, we obtain, by compositions of maps, a surjective character $\chi$ in $\Hom_{\Z_p}^{\textrm{cont}}(\npg, \Z/p^n \Z)$:

$$\chi: \npg\hookrightarrow K^*\stackrel{\cong}{\longrightarrow} F^*
\stackrel{\rho^K_F}{\longrightarrow}\langle u\rangle \stackrel{\cong}{\longrightarrow}\Z/p^n\Z.$$

\begin{theorem}\label{prior_thm_1}
\cite[Theorem 2.2]{Lu}
Let $n$ be a natural number. The above association $u\mapsto \chi$ induces a bijective correspondence between the conjugacy classes of order $p^n$ elements of the Nottingham group $\notty_\kappa$ and the strictly equivalent classes of surjective characters in $\Hom_{\Z_p}^{\textrm{cont}}(\npg, \Z/p^n \Z)$.
\end{theorem}

By Theorem \ref{prior_thm_1}, the conjugacy classes of order $p^n$ elements of $\notty_\kappa$ can be further  classified by the types $\langle b^{(0)}, \ldots, b^{(n-1)} \rangle$ of their corresponding strictly equivalent classes of characters. Hence, it makes sense to talk about the type of a torsion element.

\begin{theorem}\label{prior_thm_2}
\textrm{(Klopsch, Lubin)} \footnote{The result on $d_m$ is due to Klopsch \cite{Kl} and the result  on $d_{1,m}$ is an unpublished work of Lubin.} Let $q$ be the size of the finite field $\kappa$. Let $d_m$ denote the conjugacy classes of order $p$ elements  of $\notty_\kappa$ of type $\langle m\rangle$  and   $d_{1,m}$ denote the number of conjugacy classes  of order $p^2$ elements of $\notty_\kappa$ of  type $\langle 1, m \rangle$. Then $d_m=q-1$ and

\[
d_{1,m} = 
\begin{cases}
p(q-1)	 & \; \textrm{if} \; m \equiv 0 \pmodp. \\
(q-1)^2	 & \; \textrm{if}   \; m \equiv  1 \pmodp.\\
p(q-1)^2	 & \; \textrm{otherwise}. \\
\end{cases}
\]  
\end{theorem}

When $\kappa=\F_p$, the second author proved the following.

\begin{theorem}\label{prior_thm_3}
\cite{Ki}\footnote{In \cite{Ki}, the notation $d_{2,m}$ and $d_{2,m}^{\textrm{weak}}$ are respectively
denoted by $d_m$ and $d_{m}^{\textrm{weak}}$.} 
Let $d_{2,m}$ denote the number of conjugacy classes of  order $p^2$ elements of $\notty$  of type $\langle 2, m \rangle$ and $d_{2,m}^{\textrm{weak}}$ denote the number of 
weakly equivalent classes of surjective characters in $\Hom_{\Z_p}^{\textrm{cont}}(\npg, \Z/p^2 \Z)$
of type $\langle 2,m\rangle$.
Then
\[
d_{2,m}^{\textrm{weak}} \leq d_{2,m} \leq p d_{2,m}^{\textrm{weak}},
\]
where
\[
d_{2,m}^{\textrm{weak}} = 
\begin{cases}
p(p-1)	 & \; \textrm{if} \; m \equiv 0 \pmodp. \\
(p-1)^2	 & \; \textrm{if}   \; m \equiv  1 \pmodp.\\
p(p-1)^2	 & \; \textrm{otherwise}. \\
\end{cases}
\]  
\end{theorem}
In this article, our main result Corollary \ref{main_cor} improves the result of the second author above, namely  that we obtain $d_{2,m} = p d_{2,m}^{\textrm{weak}}$.

\section{Preliminary computations}\label{prelim_comp}
From now on till the end of the paper, we assume $K=\laurentseries$ and work on the Nottingham group $\notty$ over $\F_p$.
Consider the $\Z/p^2 \Z$-module $\Hom_{\Z_p}^{\textrm{cont}}(U_1, \Z/p^2 \Z)$ of characters.
The set
$$\{E_j:=1+t^j\in U_1: \hspace{.1in} j\in\N\}$$
contains the subset  
$$\{E_j:=1+t^j\in U_1: \hspace{.1in} j\in\N, \hspace{.1in} p\nmid j\},$$
which is a topological $\Z_p$-basis
of $U_1$ and the dual basis 
is denoted by
$$\{\FZ_j\in \Hom_{\Z_p}^{\textrm{cont}}(U_1, \Z/p^2 \Z):\hspace{.1in} j\in\N,\hspace{.1in} p\nmid j \},$$ 
i.e., $\FZ_j( E_i) = \delta_{ij}$, the Kronecker delta function. 
Let $\CalC = \{ 0, 1 \ldots, p^2 -1 \}$ be a fixed choice of representatives of $\Z/p^2 \Z$. Then any surjective character $\chi\in \Hom_{\Z_p}^{\textrm{cont}}(U_1, \Z/p^2 \Z)$ with break sequence $\langle l,m \rangle$ has the expression of the form
\begin{equation}\label{basic}
\chi = \sum_{ \substack{1 \leq j \leq m, p \nmid j}} c_j \FZ_j, 
\end{equation}
where the coefficients $c_j \in \CalC$.

%\begin{definition}\label{bsdefn} (The definition is not necessary) Let $\chi : \npg \to \Z/p^n \Z$ be a surjective continuous character. The \textit{break sequence} of the character $\chi$ is the tuple $\langle b^{(0)}, \ldots, b^{(n-1)} \rangle$ where for each $j$, the number $b^{(j)}$ is the largest integer $b$ for which there is a $z \in 1 + \nmaxi^b$ such that $\chi(z) = p^{j}$.
%\end{definition}

\begin{proposition}
\label{bs_lemma}
(Lubin \cite{Lu}) Let $\chi : \npg \to \Z/p^n \Z$ be a surjective continuous character. Let  $\langle b^{(0)}, \ldots, b^{(n-1)} \rangle$ be its break sequence. Then the following conditions hold:
\smallskip
\begin{enumerate}[(a)]
    \item\label{pri_cond} $\gcd(p,b^{(0)}) =1$;
    \smallskip
    \item\label{ine_cond} for each $i >0$, $b^{(i)} \geq p b^{(i-1)}$, and
    \smallskip
    \item\label{str_cond} if the above inequality is strict, then $\gcd(p,b^{(i)}) =1$.
\end{enumerate}
\smallskip
Conversely, every sequence $\langle b^{(0)}, \ldots, b^{(n-1)} \rangle$ satisfying the above three conditions is the break sequence of some character $\chi$ on $\npg$. There are only finitely many different characters $\chi$ with the break sequence $\langle b^{(0)}, \ldots, b^{(n-1)} \rangle$, and a \textrm{fortiori}, only finitely many strict equivalence classes of such characters.
\end{proposition}

Let $\chi$ be the character in \eqref{basic}. It follows 
from Proposition \ref{bs_lemma}
that
\smallskip
\begin{enumerate}
\label{coe_pro}
\item
$l$ is relative prime to $p$;
\smallskip
\item
$c_l$ is relative prime to $p$;
\smallskip
\item 
for all $l+1 \leq j \leq m$ such that $p \nmid j$, the coefficient  $c_j \in \Set{0,p,2p, \ldots, (p-1)p}$;
\smallskip
\item
if $m$ is relative prime to $p$, then $c_m \neq 0$.
\end{enumerate}
\smallskip
\noindent Writing $c_1,\ldots, c_l$ in the form $x + p . a$ where $x,a \in \{ 0, 1 \ldots, p-1 \}$, and changing the notation, the above expansion takes the following form
\begin{equation}\label{sta_expansion}
\chi = \sum_{1 \leq i \leq l, p \nmid i} x_i \FZ_i  + \sum_{\substack{1 \leq j \leq m,  p \nmid j}} a_j. p \FZ_j,
\end{equation}
where now $x_1, x_2, \ldots x_l, a_1, \ldots, a_m   \in \{ 0, 1, \ldots p-1 \}$, and $x_l \neq 0$, and if $(m,p) =1 $ then $a_m \neq 0$ too. 
We call \eqref{sta_expansion} the \textit{standard expansion} of $\chi$.

Let $\psi$ be a character with break sequence $\langle l \rangle$. Then $\psi$ is trivial on $1 + \maxideal^{l+1}$ but not on $1 + \maxideal^{l}$, and that $\psi$ restricts to a nonzero linear functional $\wt{\psi}$ on the one dimensional $\F_p$-vector space $(1 + \maxideal^l)/ (1 + \maxideal^{l+1})$.

\begin{theorem}\label{lubin_lemma}
(Lubin) Let $\chi,\psi: \npg \to \Z/p\Z$ be characters both of type $\langle l \rangle$. Then $\chi \simeq \psi$ if and only $\wt{\chi} = \wt{\psi}$.
\end{theorem}
\begin{proof}
For a proof see \cite[Theorem 4.2]{Lu}
\end{proof}

\begin{lemma}\label{lubin_lemma_2} \footnote{The second author thanks Jonathan Lubin for sharing the proof with him.}
Let $\chi, \psi : \npg \to \zmodpnz$ be two  characters. Then $\chi$ is strictly equivalent to $\psi$ if and only if there exists $u \in \notty$ such that $\psi = {}_u \chi$ and $\chi( u(t)/t ) \equiv 0 \pmod{p^{n-1}}$.
\end{lemma}
\begin{proof}
For a proof see \cite[Lemma 4.2]{Ki}.
\end{proof}

\begin{remark}\label{degenerate}
A variant of the following lemma can be found in \cite[Lemma 5.1]{Ki}. As a preliminary, we note that if $m \equiv 0 \pmodp$, then $m = l p$ by Proposition \ref{bs_lemma} \eqref{str_cond}.
\end{remark}

\begin{lemma}\label{evaluation}
Let $\chi, \psi$ be two characters of type $\langle l, m \rangle$ with standard expansions as in equation \eqref{sta_expansion}. Let $u(t) = t(1 + \alpha t + \beta t^2 \cdots) \in \notty$. The
following assertions hold.
\smallskip
\begin{enumerate}[(a)]
\item\label{12}
$\ub{u} \chi(E_l) \equiv x_l \equiv \chi(E_l)$  $\pmodp$.
\smallskip
\item\label{m}
If $p \nmid m$ then $\ub{u} \chi(E_m) = p\cdot a_m $.
\end{enumerate}
\end{lemma}

\begin{proof}

\begin{enumerate}[(a)]
\item
Since $E_l \circ u(t) = 1 + u(t)^l = 1 + t^l(1 + \alpha t + \beta t^2 + \cdots)^l = 1 + t^l +l \alpha t^2 + \cdots  = (1 + t^l) \cdots$ we have, modulo $p$,
\begin{align*}
\ub{u} \chi (E_l)
&= \chi(E_l \circ u(t)) \equiv \chi((1+t)^l) \\
&=  \chi(E_l) \equiv x_l. 
\end{align*}

\item
 If $p \nmid m$ then we have
\begin{align*}
{}_{u} \chi (E_m )  =  \chi( E_m \circ u(t)) =  \chi (1 +  t^m) = \chi(E_m) = p\cdot a_m  .
\end{align*}
\end{enumerate}
\end{proof}

\section{Classification of strict equivalence classes} \label{classification}
\begin{lemma}\label{lem7}
Let $\chi$ be a character of type $\langle l , m \rangle$ with the standard expansion \eqref{sta_expansion}. Then $\chi$ is strictly equivalent to the character with standard expansion 
\begin{equation}\label{row2}
 x_l \FZ_l  + \sum_{\substack{1 \leq j \leq m, p \nmid j}} b_j. p \FZ_j,  
\end{equation}
where  $b_j \in \{0,1,...,p-1\}$. In addition, if $p \nmid m$ then $b_m = a_m$.
\end{lemma}

\begin{proof}
By Theorem \ref{lubin_lemma},
$\chi \pmodp \simeq x_l \FZ_\ell\pmodp$, so that, by definition, there exists an element $u\in\notty$ such that 
$x_l \FZ_\ell\equiv {}_u \chi\pmodp$ and $\chi(u(t)/t)\equiv 0\pmodp$. By Lemma \ref{lubin_lemma_2} for $n=2$,
$\chi$ is strictly equivalent to ${}_u \chi$ and the latter is of the form \eqref{row2}.
The last assertion follows from Lemma \ref{evaluation}(b).
\end{proof}

The form \eqref{row2} is strictly equivalent to a even more simple form.

\begin{proposition}\label{prop7}
Let $\chi$ be a character of type $\langle l , m \rangle$ with the standard expansion \eqref{row2}.
Then $\chi$ is strictly equivalent to a character with the standard expansion
\begin{equation}\label{row20}
x_l \FZ_l + \sum_{ m-l \leq j \leq m, p \nmid j} b_j \cdot p \FZ_j,
\end{equation}
where the coefficients 
$x_l$ and $b_j$ for
$m-l \leq j \leq m$, $p \nmid j$
are the coefficients in \eqref{row2}.
\end{proposition}

\begin{proof}
 Let $\chi$ be the character of type $\langle l , m \rangle$ in \eqref{row2}.  Then there exists $b_m\in\{1,...,p-1\}$ such that 
$$b_m\cdot p =\chi(E_m).$$ 
By induction, it suffices to find, for each index $j$ starting from $1$ to $m-l-1$  with $m-l-j$ relative prime to $p$, an element $u_j:=u_j(t)\in\notty$ satisfying three conditions:
\smallskip
\begin{enumerate}[(a)]
\item ${}_{u_j} \chi(E_{m-l-j})=0;$
\smallskip
\item $\chi(u_j(t)/t)=0$; 
\smallskip
\item $\chi(1+t^k)=\chi(1+u_j^k)$
if $m-l-j<k\leq m$.
\smallskip
\end{enumerate}

Consider 
$$u_j :=u_j(t)=t(1+t^{l+j})^d(1+t^m)^e\in\notty,$$ 
where $d$ and $e$ are integers to be chosen later. Then, we obtain
\begin{align*}
1 + u_j^{m-l-j} 
&= 1 + t^{m-l-j}(1 + t^{l+j})^{d (m-l-j)} (1 + t^m)^{e (m-l-j)} \\
&= 1 + t^{m-l-j} + d (m-l-j)t^m  \pmod{U_{m+1}} \\
&= (1 + t^{m-l-j})(1 + t^m)^{d (m-l-j)} \pmod{U_{m+1}}
\end{align*}

Therefore 
\begin{align}\label{compute}
\begin{split}
\ub{u_j} \chi(E_{m-l-j}) &= \chi(1 + u_j^{m-l-j}) \\
& = p\cdot (b_{m-l-j} + d(m-l-j)  b_m )
\end{split}
\end{align}
Since $b_m$ and $m-l-j$ are both relative prime to $p$, there exists some $d\in\N$ such that \eqref{compute} is zero in $\Z/p^2\Z$, i.e., (a) is fulfilled. 
On the other hand, as $l+j>l$ we have
\begin{align}\label{compute2}
\begin{split}
\chi(u_j(t)/t) &= \chi((1 + t^{l+j})^d (1 + t^m)^e) \\
&= p\cdot( d b_{l+j} + e b_m).
\end{split}
\end{align}
Again since $b_m$ is prime to $p$, there exists $e\in\N$ such that \eqref{compute2} is zero in $\Z/p^2\Z$, i.e., (b) is fulfilled.
Finally, one checks easily that (c) is also fulfilled by the definition of $u_j$. 
We are done.
\end{proof}

\begin{customthm} {\ref{main_thm}}
Let $[\chi]$ be a strict equivalence class of type $\langle l, m \rangle$.  Then there exists a character in the class $[\chi]$ whose standard expansion is 
\begin{equation*}
x_l \FZ_l + \sum_{ m-l \leq j \leq m, p \nmid j} b_j \cdot p \FZ_j, 
\tag{\ref{intro_reduction}}
\end{equation*}
where $x_l \neq 0$ and
if $p \nmid m$, then $b_m \neq 0$. 
The character in \eqref{intro_reduction} is said to be in \textit{reduced form}.
\end{customthm}

\begin{proof}
The assertion follows directly from Lemma \ref{lem7} and Proposition \ref{prop7}.
\end{proof}

\begin{corollary}\label{main_cor}
Let $B(p,l,m)$ be  the number of type $\langle l, m \rangle$ reduced forms \eqref{intro_reduction}.
The number of strict equivalence classes of type $\langle l, m \rangle$ 
 is at most $B(p,l,m)=p^k (p-1)^\epsilon$, where $k$ is the number of integers in $[m-l,m-1]$ incongruent to $0$ modulo $p$, and $\epsilon$ is equal to $1$ if $m$ is congruent to $0$ modulo $p$, and equal to $2$ otherwise.
\end{corollary}

\begin{proof}
If $m$ is incongruent to $0$ modulo $p$, then each of $x_l, b_m$, being nonzero, can be chosen in $p-1$ different ways, otherwise
when $m$ is congruent to $0$ modulo $p$ the coefficient $a_m$ does not appear in the
standard expansion and so that only $x_l$ needs to be chosen, which can be done in  $p-1$ ways. The assertion about the other factor $p^k$, with restriction on $k$, is evident.
\end{proof}

\begin{customthm}{\ref{exact}}
Suppose $l < p$. Then each strict equivalence class of type $\langle l, m \rangle$ has precisely one representative in the reduced form \eqref{intro_reduction} in Theorem \ref{main_thm}.
\end{customthm}

\begin{proof}
Suppose $l<p$ and $\chi$ is in reduced form  \eqref{intro_reduction}. Let $u=u(t)\in\notty$ satisfy 
\begin{equation}\label{kern}
\chi(u(t)/t)=0
\end{equation}
and that
\begin{equation}\label{long}
\chi' = \ub{u} \chi = x'_l \FZ_l + \sum_{ m-l \leq j \leq m, p \nmid j} b_j' \cdot p \FZ_j
\end{equation}
is also in reduced form but distinct from  $\chi$. By Lemma \ref{evaluation}(a) and (b), we obtain $x_l = x'_l$ and if $p \nmid m$, then $b_m = b_m'$. If $\widetilde{j}$ denotes the the largest index $j$ such that $b_j\neq b'_j$, then $\widetilde{j}$ is strictly less than $m$, so that  $0<m-\widetilde{j}\leq l$. Write 
$$u(t)=t\prod_{k=1}^\infty (1+t^k)^{n_k}$$
where $n_k\in\{0,1,...,p-1\}$. If  
$\widetilde{k}$ denotes the smallest index $k$ such that 
$n_k  \neq 0$, then the inequalities $0<\widetilde{k}\leq m-\widetilde{j}$ hold; indeed if $\widetilde{k} > m - \widetilde{j}$, then $b_{\widetilde{j}} = b'_{\widetilde{j}}$ which contradicts the definition of $\widetilde{j}$. Hence, we conclude that
$$0<\widetilde{k}\leq  m-\widetilde{j}\leq l<p.$$

The equality $\widetilde{k}=l$ is impossible since it contradicts \eqref{kern}.
It follows that
$0< l-\widetilde{k}<p$, and together with  $x_l,n_{\widetilde{k}}\in \{1,...,p-1\}$, we obtain
\begin{align*}
\begin{split}
\chi'(E_{l-\widetilde{k}})=\ub{u}
\chi(1+t^{l-\widetilde{k}})
&=\chi(1+t^{l-\widetilde{k}}(1+n_{\widetilde{k}}t^{\widetilde{k}}+\cdots)^{l-\widetilde{k}})\\
&=\chi(1+t^{l-\widetilde{k}}+n_{\widetilde{k}}(l-\widetilde{k}) t^l+\cdots)\\
&=\chi((1+t^{l-\widetilde{k}})(1+ t^l)^{n_{\widetilde{k}}(l-\widetilde{k})}\cdots)\\
&= n_{\widetilde{k}}(l-\widetilde{k})x_l\neq 0 ~\text{mod $p$},
\end{split}
\end{align*}
contrary to the hypothesis that $\chi'$ is in reduced form.
\end{proof}

\begin{remark}
Consider $p=2$ and the type $\langle 5,15\rangle$ character $\chi$ in reduced form:
$$\chi= \FZ_5 + 2\FZ_{15}.$$
Take $u(t)=t(1+t^3+t^4)(1+t^{15})^e \in\notty$ so that $\chi(u(t)/t)=0$. 
Then ${}_u \chi$ is of the form \eqref{row2} such that ${}_u \chi(E_{11})=2$. By Proposition \ref{prop7}, ${}_u \chi$ is strictly equivalent to a reduced form $\psi$ with $\psi(E_{11})=2\neq 0= \chi(E_{11})$. 
But $\psi$ and $\chi$ are strictly equivalent and both are in reduced forms.
\end{remark}

\section{Torsion elements of $\notty$ of order $p^2$} \label{torsion_elements}

\begin{customthm}{\ref{intro_final_thm}}
\label{final_thm}
Let $d_{l,m}$ denote the number of conjugacy classes of elements of the Nottingham group $\notty$ that are of order $p^2$ and type $\langle l, m \rangle$. Then
\begin{equation*}
d_{l,m} \leq B(p,l,m)=p^k (p-1)^\epsilon
\tag{\ref{ineq}}
\end{equation*}
where $k$ is the number of integers in $[m-l,m-1]$ incongruent to $0$ modulo $p$, and $\epsilon$ is equal to $1$ if $m$ is congruent to $0$ modulo $p$, and equal to $2$ otherwise. The inequality \eqref{ineq} is an equality if $l<p$.
\end{customthm}

\begin{proof}
The result is an immediate consequence of Theorem \ref{prior_thm_1} of Lubin, Corollary \ref{main_cor},
and Theorem \ref{exact}.
\end{proof}

\begin{remark}
Since $l<p$ always holds when $l=1,2$, Theorem \ref{intro_final_thm} 
recovers the formula of $d_{1,m}$
in Theorem \ref{prior_thm_2}
when $\kappa=\F_p$
and 
is an improvement over Theorem \ref{prior_thm_3} of the second author in which only the upper bound on $d_{2,m}$ was established.
\end{remark}

\begin{customthm}{\ref{conjugate}}
Let $n\in\N$ and $u\in\notty$ be a torsion element of order $p^2$ and type $\langle l,m\rangle$. If
$u$ and $u^n$ are distinct elements,
then they are conjugate in $\notty$ if and only if $n\equiv 1 \pmodp$ and $(p,l,m)\neq (2,l,2l)$.
\end{customthm}

\begin{proof}
Let $\chi$ be a surjective character of type $\langle l,m\rangle$ associated to $u$.\\

($\Rightarrow$)
Suppose $u$ and $u^n$ are conjugate in $\notty$. It follows that $n$ is prime to $p$.
Then one sees from the correspondence in $\mathsection2$ that $n\cdot\chi$ is a surjective character associated to $u^n$.
Let $x_l\neq0$ be $\chi(E_l)\pmodp$. As $u$ and $u^n$ are conjugate, $\chi$ and $n\cdot\chi$ are strictly equivalent. Hence by Lemma \ref{evaluation}(a), we have
$$x_l\equiv n x_l\pmodp$$
which implies $n\equiv 1 \pmodp$.

On the other hand, assume $(p,l,m)=(2,l,2l)$ for some odd $l$, we need to prove that $u$ and $u^n$ are not conjugate.
Let $\chi_1$ be a reduced form of $u$.
Since $m-l=l$ is prime to $p=2$, the character 
$n\cdot \chi_1$ is a reduced form of $u^n$.
If $u$ and $u^n$ are conjugate, then 
$n$ is odd.
Since $u\neq u^n$ and $p^2=4$, we may assume $n=3$.
It follows that the reduced forms $\chi_1 \neq 3\chi_1$ differ only at the coefficient $b_l$ (see \eqref{row20}) and 
${}_w \chi_1 = 3\chi_1$ for some 
$w\in\notty$ satisfying
\smallskip
\begin{enumerate}[(a)]
    \item $w:=w(t)\neq t$;
    \item $\chi_1(w(t)/t)=0$.
\end{enumerate}  
\smallskip
The conditions (a) and (b) imply that $0<k< l$ if we write $w(t)=t(1+t^k+\cdots)\in\notty$.
If $k$ is odd, then one computes by 
using $\chi_1 (E_{2l})\neq 0$ and $\chi_1$ is a reduced form that
$${}_w \chi_1 (E_{2l-k})=\chi_1 (E_{2l-k}) + \chi_1 (E_{2l})\neq  \chi_1 (E_{2l-k})=3\chi_1 (E_{2l-k}).$$
If $k$ is even, then one computes by using $\chi_1$ is a reduced form that
$${}_w \chi_1 (E_{l-k})\equiv
\chi_1 (E_{l})\not\equiv 0 \equiv 3\chi_1 (E_{l-k}) ~~(\mathrm{mod}~2).$$
Since both cases contradict the equation ${}_w \chi_1 = 3\chi_1$,
$u$ and $u^n$ are not conjugate.\\

($\Leftarrow$) Without loss of generality, assume $\chi$ is in reduced form. 
Then $n\cdot \chi$ is in the form of \eqref{row2}.
By Proposition \ref{bs_lemma} and  Remark \ref{degenerate},
$(p,l,m)\neq (2,l,2l)$ is equivalent to $m-l>l$. 
Apply Proposition \ref{prop7} to $n\cdot \chi$ and we obtain a reduced form $\psi\simeq n\cdot \chi$. Since  $m-l>l$ and 
$n\equiv 1 \pmodp$,
it follows that $\psi$ is identical to $\chi$, which implies that $n\cdot \chi\simeq \chi$. Therefore, $u^n$ and $u$ are conjugate in $\notty$.
\end{proof}

\bigskip

\begin{center}
\begin{tabular}{ll} 
Chun Yin Hui  & Krishna Kishore \\ 
Yau Mathematical Sciences Center    & Dept. of Mathematics \\
Tsinghua University  & Indian Institute of Technology-Dharwad\\
Haidian District, Beijing 100084,  &  WALMI Campus, Dharwad\\
China  & Karnataka 580011, India\\
\smallskip
\email{pslnfq@gmail.com} & \email{kishore@iitdh.ac.in}
\end{tabular}
\end{center}

\end{document}